\documentclass[reqno]{amsart}
\usepackage{graphicx,graphics}
\usepackage{subfig}

\newcommand{\D}{\mathcal{D}}

\newcommand{\Real}{\mathbb R}
\newcommand{\abs}[1]{\left\vert#1\right\vert}

\theoremstyle{definition}

\newtheorem{lem}{Lemma}
\newtheorem{thm}{Theorem}

\newtheorem{defn}{Definition}

\numberwithin{thm}{section}
\numberwithin{lem}{section}
\numberwithin{coll}{section}
\numberwithin{rem}{section}
\numberwithin{exm}{section}
\numberwithin{prop}{section}
\numberwithin{equation}{section}
\numberwithin{equation}{section}

\begin{document}
\centerline {\textsc {\large Rates of convergence for R\'{e}nyi entropy in extreme value theory}} 
\vspace{0.5in}
\begin{center}
   Ali Saeb\footnote{Corresponding author: ali.saeb@gmail.com}\\
Theoretical Statistics and Mathematics Unit,\\ Indian Statistical Institute, Delhi Center,\\7 S.J.S Sansanwal Marg, New Delhi 110016, India
\end{center}
\vspace{1in}

\noindent {\bf Abstract:}
Max stable laws are limit laws of linearly normalized partial maxima of independent identically distributed random variables. Saeb (2014) proves that the R\'{e}nyi entropy of order $\beta$ ($\beta>1$) of linear normalized maximum of iid random variables with continuous differentiable density is convergent to the R\'{e}nyi entropy of order $\beta$ of the max stable laws. In this paper, we study the rate of convergence result for R\'{e}nyi entropy for linearly normalized partial maxima.

\vspace{0.5in}

\vspace{0.2in} \noindent {\bf Keywords:} Rate of convergence, R\'{e}nyi entropy, Densities convergence, Max stable laws, Max domain of attraction.

\vspace{0.5in}

\vspace{0.2in} \noindent {\bf MSC 2010 classification:} 60F10
\newpage
\section{Introduction}
The Shannon entropy of a continuous random variable (rv) $X$ with  density function $f(x)$ is defined as 
\begin{eqnarray*}
	H(f) =   -\int_A f(x)\log f(x)dx, \;\;\mbox{where}\;\;A=\{x\in\Real:f(x)>0\}.
\end{eqnarray*}
 R\'{e}nyi (1961) generalizes the differential entropy as the following which is called the R\'{e}nyi entropy of order $\beta :$
\begin{equation}\label{Renyi_entropy}
H_\beta(f) = \frac{1}{1-\beta}\log\left(\int_A (f(x))^\beta\,dx\right) ,
\end{equation}
where, $0<\beta<\infty,$ $\beta\neq 1.$ The R\'{e}nyi entropy of order  $2,$ $H_2(f),$ is called the collision entropy. It is to be noted here that, as $\beta\to 1,$ the R\'{e}nyi entropy tends to Shannon entropy, which can be seen as the negative expected log likelihood. 

 The idea of tracking the central limit theorem using Shannon entropy goes back to Linnik (1959) and Shimizu (1975), who used it to give a particular proof of the central limit theorem. Brown (1982), Barron (1986) and Takano (1987) discuss the central limit theorem with convergence in the sense of Shannon entropy and relative entropy. Artstein et al.(2004) and Johnson and Barron (2004) obtained the rate of convergence under some conditions on the density. Johnson (2006) is a good reference to the application of information theory to limit theorems, especially the central limit theorem. Cui and Ding (2010) show  that the convergence of the Renyi entropy of the normalized sums of iid rvs and obtain the corresponding rates of convergence.
Saeb (2014) study the Renyi entropy of the max domain of attraction (MDA) with continuous density is convergent to the Renyi entropy of the max stable laws.

The limit laws of linearly normalized partial maxima $M_n=\max(X_1,\cdots, X_n)$ of independent and identically distributed (iid) rvs $X_1,X_2,\ldots,$ with common distribution function (df) $F,$ namely,
\begin{equation}\label{Introduction_e1}
	\lim_{n\to\infty}\Pr(M_n\leq a_nx+b_n)=\lim_{n\to\infty}F^n(a_nx+b_n)=G(x),\;\;x\in \mathcal{C}(G),
\end{equation}
	where, $a_n>0,$ $b_n\in\Real,$ are norming constants, $G$ is a non-degenerate distribution function, $\mathcal{C}(G)$ is the set of all continuity points of $G,$ are called max stable laws. If, for some non-degenerate distribution function $G,$ a distribution function $F$ satisfies (\ref{Introduction_e1}) for some norming constants $a_n>0,$ $b_n\in\Real,$ then we say that $F$ belongs to the MDA of $G$ under linear normalization and denote it by $F\in \mathcal{D}(G).$ Limit distribution functions $G$ satisfying (\ref{Introduction_e1}) are the well known extreme value types of distributions, or max stable laws, namely,
	\begin{eqnarray*}
			\text{the Fr\'{e}chet law:} & \Phi_\alpha(x)  = \left\lbrace	
							\begin{array}{l l}
							 0, &\;\;\; x< 0, \\
							 \exp(-x^{-\alpha}), &\;\;\; 0\leq x;\\
							 \end{array}
							 \right. \\					
		\text{the Weibull law:} & \Psi_\alpha(x) =  \left\lbrace
						\begin{array}{l l} \exp(- |x|^{\alpha}), & x<0, \\
						1, & 0\leq x;
						\end{array}\right. \\
		\text{and the Gumbel law:} & \Lambda(x) = \exp(-\exp(-x));\;\;\;\;\; x\in\Real;
	\end{eqnarray*}
$\alpha>0$ being a parameter, with respective density functions,
	\begin{eqnarray*}
		\text{the Fr\'{e}chet density:} & \phi_\alpha(x) =  \left\lbrace	
							\begin{array}{l l}
							 0, &\;\;\; x \leq 0, \\
							 \alpha x^{-(\alpha+1)}e^{-x^{-\alpha}}, &\;\;\; 0 < x;\\
							 \end{array}
							 \right. \\
		\text{the Weibull density:} & \psi_\alpha(x) = \left\lbrace
						\begin{array}{l l} \alpha |x|^{\alpha-1}e^{-|x|^{\alpha}}, & x<0, \\
						0, & 0 \leq x;
						\end{array}\right. \\
		\text{and the Gumbel density:} & \lambda(x) = e^{-x}e^{-e^{-x}},\;\; x\in\Real.
	\end{eqnarray*}
Note that (\ref{Introduction_e1}) is equivalent to
\begin{eqnarray}
\lim_{n\to\infty}n(1-F(a_nx+b_n))=-\log G(x), \; x \in \{y: G(y) > 0\}.\nonumber
\end{eqnarray}
We shall denote the left extremity of distribution function $F$ by $l(F) = \inf\{x: F(x) > 0\} \geq - \infty $ and the right extremity of $F$ by $r(F) = \sup \{x: F(x) < 1\} \leq \infty.$
Criteria for $F\in \mathcal{D}(G)$ are well known (see, for example, Galambos, 1987; Resnick, 1987; Embrechts et al., 1997). de Haan and Resnick (1982), show that under von Mises conditions the density of the normalized maximum convergence to the limit density in $L_p$ ($0<p\leq\infty$) provided both the orginal denisty and the limit density are in $L_p.$
Smith (1982) derive the rate of convergence in distribution of normalized sample maxima to corresponding max stable laws. Omey (1988) study the rate of convergence for density of MDA of Fr\'{e}chet to the appropriate limit density.

 In this article, our main interest is to investigate the rate convergence of Renyi entropy for normalized partial maxima of iid rvs corresponding to Renyi entropy of max stable laws. Since, the rate of convergence on Renyi entropy in MDA involve the rate of convergence of density function of sample maxima, the proofs of our results here involve the application of rate convergence of density function on MDA. In section 2, we study the rate of convergence for density function of MDA of Gumbel and also, we represent the new form of Theorem 6, Omey (1988) to find out the new conditions on rate of convergence in MDA of Fr\'{e}chet. Section 3, we give our main result on rate of convergence for Renyi entropy for MDA corresponding to max stable laws. For quick reference, some of the results used in this article are given in Appendix \ref{more results}.

\section{Rates of convergence for density function}
\subsection{Gumbel domain} 
Suppose $F \in \mathcal{D}(\Lambda)$ and $1 - F$ is $\Gamma$ varying so that
\begin{eqnarray} \label{s_RegVar_g}
\lim_{n \rightarrow\infty} \frac{\overline{F}(b_n+xa_n)}{\overline{F}(b_n)} & = &  e^{-x}, \;\; x\in\Real,\nonumber
\end{eqnarray}
where the function $a_n=u(b_n)=\int_{b_n}^{r(F)}\overline{F}(s)ds/\overline{F}(b_n)$ is called an auxiliary function. From (\ref{Introduction_e1})
\begin{eqnarray}
\lim_{n \rightarrow \infty} F^n(a_n x+b_n) & = & \Lambda(x), \; x \in\Real,  \nonumber
\end{eqnarray}
with $\;a_n = u(b_n)$ and $b_n=F^{\leftarrow}(1 - \frac{1}{n}) = \inf\{x: F(x) > 1 - \frac{1}{n}\}, \, n \geq 1.\;$
From Theorem \ref{thm_g_von}, $F$ satisfies the von Mises condition:
\begin{equation} \label{s_vonMises_g}
\lim_{t \rightarrow r(F)} \dfrac{f(t)u(t)}{\overline{F}(t)} = 1.
\end{equation}

We define $F(x)=\exp\{-e^{-\eta(x)}\}$ and $F$ is twice differentiable (see, page 114, Resnick (1987)).
The Von Mises condition analogous to (\ref{s_vonMises_g}) guaranteeing $F\in\D(\Lambda)$ is
\begin{eqnarray}
h(x)=(1/\eta')'=-\log F(x)-\left[1-\frac{F(x)f'(x)\log F(x)}{(f(x))^2}\right]\to 0.\label{G_h}
\end{eqnarray}
as $x\to r(F).$ There exists a nonincreasing function $h$ with $h(x)=sup_{x<y}h(y)$ note that $h(y)\downarrow 0,$ as $y\to r(F).$
From equation (2.54) and (2.55), Resnick (1987), obtain a following bound on $x>0,$
\begin{eqnarray}
F(h(b_n),x)\leq F^n(a_nx+b_n)\leq F(-h(b_n),x).\label{G_Fe1}
\end{eqnarray}
and for region $(-\infty,0),$
\begin{eqnarray}
F(h(a_nx+b_n),x)\leq F^n(a_nx+b_n)\leq F(-h(a_nx+b_n),x).\label{G_Fe2}
\end{eqnarray}
\begin{lem}\label{G_u}
Suppose $F\in\D(\Lambda)$ and (\ref{G_h}) hold, then
\begin{eqnarray}
u(a_nx+b_n)=-\frac{F(a_nx+b_n)}{f(a_nx+b_n)}\log F(a_nx+b_n).\nonumber
\end{eqnarray}
\end{lem}

\begin{proof} Let $\log F(a_nx+b_n)=-e^{-\eta(a_nx+b_n)}.$ We have
\begin{eqnarray}
\frac{f(a_nx+b_n)}{F(a_nx+b_n)}=\eta'(a_nx+b_n) e^{-\eta(a_nx+b_n)}.\nonumber
\end{eqnarray}
Since, $u(a_nx+b_n)=\frac{1}{\eta'(a_nx+b_n)},$ we have
\begin{eqnarray}
u(a_nx+b_n)=-\frac{F(a_nx+b_n)\log F(a_nx+b_n)}{f(a_nx+b_n)}.\nonumber
\end{eqnarray}
\end{proof}

\begin{lem}\label{G_lem1} If (\ref{G_h}) hold and $\abs{u'(t)}<h(t),$ then 
\begin{itemize}
\item[i.] for $x>0,$
\[\frac{1}{1+h(b_n)x}\leq \frac{u(b_n)}{u(a_nx+b_n)}\leq \frac{1}{1-h(b_n)x}.\]
\item[ii.] and for $x<0,$
\[\frac{1}{1+h(a_nx+b_n)x}\leq \frac{u(b_n)}{u(a_nx+b_n)}\leq \frac{1}{1-h(a_nx+b_n)x}.\]

\end{itemize}
\end{lem}
\begin{proof}
i. Recalling that $u(x)=\frac{1}{\eta'(x)}$ we have for $x>0,$
\begin{eqnarray}
\abs{\frac{u(a_nx+b_n)-u(b_n)}{u(b_n)}}\leq \int_{b_n}^{a_nx+b_n}\frac{u'(t)dt}{u(b_n)}\leq \frac{h(b_n)a_n x}{u(b_n)}=h(b_n)x.\nonumber
\end{eqnarray}
Therefore,
\begin{eqnarray}
1-h(b_n)x\leq \frac{u(a_nx+b_n)}{u(b_n)}\leq 1+h(b_n)x.\nonumber
\end{eqnarray}
ii. For $x<0,$
\begin{eqnarray}
\abs{\frac{u(b_n)-u(a_nx+b_n)}{u(b_n)}}\leq \int_{a_nx+b_n}^{b_n}\frac{u'(t)dt}{u(b_n)}\leq \frac{h(a_nx+b_n)a_n x}{u(b_n)}=h(a_nx+b_n)x.\nonumber
\end{eqnarray}
Therefore,
\begin{eqnarray}
1-h(a_nx+b_n)x\leq \frac{u(a_nx+b_n)}{u(b_n)}\leq 1+h(a_nx+b_n)x.\nonumber
\end{eqnarray}
\end{proof}

\begin{thm}\label{con_G}
Suppose $F\in\D(\Lambda)$ and there exist $k>0$ such that for $n\geq n_0$
\begin{eqnarray}
\frac{u(b_n)}{b_nh(b_n)}\leq k,\nonumber
\end{eqnarray}
and for $c<k^{-1}$ and $b>0,$ $d>0,$ 
\begin{eqnarray}
\frac{h(b_n(1-ck))}{h(b_n)}\leq d(1-ck)^{-b},\nonumber
\end{eqnarray}
and (\ref{G_h}) holds. If $\sup_{x\in\Real}f(x)<\infty$ then for large $n$ 
\begin{eqnarray}
\sup_{x\in\Real} \abs{g_n(x)-\lambda(x)}&<&O(h(b_n))+ O(\log(1+h(b_n)))+O\left(ne^{-n^{1/2}}\right).\nonumber
\end{eqnarray}
\end{thm}
\begin{proof}We have
\begin{eqnarray}
\sup_{x\in\Real}\abs{g_n(x)-\lambda(x)}&<&\sup_{x\in\Real}
\left(na_n f(a_nx+b_n)\abs{F^{n-1}(a_nx+b_n)-\Lambda(x)}\right)\nonumber\\
&&+\sup_{x\in\Real}\left(\abs{na_n f(a_nx+b_n)\Lambda(x)-\lambda(x)}\right),\nonumber\\
&<&\sup_{x\in\Real}
\left(\frac{na_n f(a_nx+b_n)}{F(a_nx+b_n)}\abs{F^{n}(a_nx+b_n)-\Lambda(x)}\right)\nonumber\\
&&+\sup_{x\in\Real}\abs{na_n f(a_nx+b_n)\Lambda(x)-\lambda(x)}.\nonumber
\end{eqnarray}
where, $\abs{F^n(a_nx+b_n)-F(a_nx+b_n)\Lambda(x)}$ is equivalent $\abs{F^n(a_nx+b_n)-\Lambda(x)}.$
From Lemma \ref{G_u} we have
\begin{eqnarray}
\sup_{x\in\Real}\abs{g_n(x)-\lambda(x)}&<&\sup_{x\in\Real}
\left(\frac{nu(b_n)\bar{F}(a_nx+b_n)}{u(a_nx+b_n)}\abs{F^n(a_nx+b_n)-\Lambda(x)}\right)\nonumber\\
&&+\sup_{x\in\Real}\left(\Lambda(x)\abs{\frac{nu(b_n)\bar{F}(a_nx+b_n)}{u(a_nx+b_n)}-n\bar{F}(b_n)}\right)\nonumber\\
&&+\sup_{x\in\Real}\left(\lambda(x)\abs{n\bar{F}(b_n)e^x-1}\right).\nonumber\\
&=&\sup_{x\in\Real}[A_1(n,x)+A_2(n,x)+A_3(n,x)].\nonumber
\end{eqnarray}
Set, $A_1(n,x)=\frac{nu(b_n)\bar{F}(a_nx+b_n)}{u(a_nx+b_n)}\abs{F^n(a_nx+b_n)-\Lambda(x)}.$ From Theorem \ref{rate_G1}, and Lemma \ref{G_lem1}-i and taking logarithm from (\ref{G_Fe1}) we have
\begin{eqnarray}
\sup_{x\geq 0}A_1(n,x)&<&\sup_{x\geq 0}\left(-\frac{\log F(h(b_n),x)}{1-h(b_n)x}O(h(b_n))\right).\nonumber
\end{eqnarray}
From Theorem \ref{rate_G2}, $-\log F(h(b_n),x)<-\log(\Lambda(x)-e^{-1}h(b_n))$ therefore 
\begin{eqnarray}
\sup_{x\geq 0}A_1(n,x)&<&\sup_{x\geq 0}\left(\frac{\log(\Lambda(x)-h(b_n)e^{-1})}{h(b_n)x-1}O(h(b_n))\right),\nonumber\\
&<&-\log(1-h(b_n)e^{-1})O(h(b_n)),\nonumber\\
&<&O(h(b_n)).\label{G_A11}
\end{eqnarray}
where, $\sup_{x\geq 0}\frac{1}{h(b_n)x-1}<-1.$

Next, we choose $\xi_n$ by $-\log F(\xi_n)\sim n^{-1/2}$ so that $\xi_n\to r(F)$ and $t_n=\frac{\xi_n-b_n}{a_n}.$ If $\frac{\xi_n-b_n}{a_n}\to c>0$ then $n^{1/2}\sim -n\log F(t_n a_n+b_n)\to e^{-c}$ and this is contradict the fact that $n^{1/2}\to\infty$ therefore $\frac{\xi_n-b_n}{a_n}\to -\infty$ for large $n.$\\

Set, $\sup_{t_n<x<0}A_1(n,x)=\sup_{t_n<x<0}\left(\frac{nu(b_n)\bar{F}(a_nx+b_n)}{u(a_nx+b_n)}\abs{F^n(a_nx+b_n)-\Lambda(x)}\right).$ From Theorem \ref{rate_G3}, if (\ref{G_e2}) and (\ref{G_e3}) hold then $\sup_{x\in\Real}\abs{F^n(a_nx+b_n)-\lambda(x)}\leq O(h(b_n))$, and Lemma \ref{G_lem1}-ii and taking logarithm from (\ref{G_Fe2}) we have
\begin{eqnarray}
\sup_{t_n<x< 0}A_1(n,x)&<&\sup_{t_n<x< 0}\left(-\frac{\log F(h(a_nx+b_n),x)}{1-h(a_nx+b_n)x}O(h(b_n))\right).\nonumber
\end{eqnarray}
From Theorem \ref{rate_G2}, $-\log F(h(a_nx+b_n),x)<-\log(\Lambda(x)-e^{-1}h(a_nx+b_n))$ therefore 
\begin{eqnarray}
\sup_{t_n<x< 0}A_1(n,x)&<&\sup_{t_n<x< 0}\left(\frac{\log(\Lambda(x)-h(a_nx+b_n)e^{-1})}{h(a_nx+b_n)x-1}O(h(b_n))\right),\nonumber\\
&<&\left(\frac{\log(e^{-1}(1-h(b_n)))}{h(b_n)t_n-1}O(h(b_n))\right),\nonumber\\
&<&O(h(b_n)).\label{G_A12}
\end{eqnarray}

From, (\ref{G_A11}) and (\ref{G_A12}) we have
\begin{eqnarray}
\sup_{t_n<x}A_1(n,x)&=&\sup_{t_n<x<0}A_1(n,x)+ \sup_{0\leq x}A_1(n,x),\nonumber\\
&<&O(h(b_n)).\label{G_A1}
\end{eqnarray}

Now we set,
\begin{eqnarray}
\sup_{t_n<x}A_2(n,x)&=&\sup_{t_n<x<0}A_{21}(n,x)+\sup_{0<x}A_{22}(n,x).\nonumber
\end{eqnarray}
We have 
\begin{eqnarray}
\sup_{t_n<x<0}A_{21}(n,x)=\sup_{t_n<x<0}\left(\Lambda(x)\abs{\frac{nu(b_n)\bar{F}(a_nx+b_n)}{u(a_nx+b_n)}-n\bar{F}(b_n)}\right),\nonumber
\end{eqnarray}
From Lemma \ref{G_lem1}-ii, we have
\begin{eqnarray}
\sup_{t_n<x< 0}A_1(n,x)&<&\sup_{t_n<x< 0}\left(-\frac{\log F(h(a_nx+b_n),x)}{1-h(a_nx+b_n)x}+n\log F(b_n)\right),\nonumber
\end{eqnarray}
where, $\sup_{t_n<x<0}\Lambda(x)\leq 1.$
 From Theorem \ref{rate_G2}, $-\log F(h(a_nx+b_n),x)<-\log(\Lambda(x)-e^{-1}h(a_nx+b_n))$ and Theorem \ref{rate_G1} we get $n\log F(b_n)<\log(e^{-1}(1+h(b_n)))$ therefore 
\begin{eqnarray}
\sup_{t_n<x< 0}A_{21}(n,x)&<&\sup_{t_n<x< 0}\left(\frac{\log(\Lambda(x)-h(a_nx+b_n)e^{-1})}{h(a_nx+b_n)x-1}+\log((1+h(b_n))e^{-1})\right),\nonumber\\
&<&O(\log((1+h(b_n))e^{-1})).\label{G_A21}
\end{eqnarray}

Similar argument we have,
\begin{eqnarray}\sup_{0<x}A_{22}(n,x)=\sup_{0<x}\left(\Lambda(x)\abs{\frac{nu(b_n)\bar{F}(a_nx+b_n)}{u(a_nx+b_n)}-n\bar{F}(b_n)}\right),\nonumber
\end{eqnarray}
From Lemma \ref{G_lem1}-i we have
\begin{eqnarray}
\sup_{0<x}A_{22}(n,x)&<&\sup_{0<x}\left(\frac{\log F(h(b_n),x)}{h(b_n)x-1}+n\log F(b_n)\right),\nonumber
\end{eqnarray}
where, $\sup_{0<x}\lambda(x)=1.$
From Theorem \ref{rate_G2}, $-\log F(h(b_n),x)<-\log(\Lambda(x)-e^{-1}h(b_n))$ and with Theorem \ref{rate_G1} we have
\begin{eqnarray}
\sup_{0<x}A_{22}(n,x)&<&\sup_{0<x}\left(\frac{\log(1-h(b_n)e^{-1})}{h(b_n)x-1}+\log((1+h(b_n))e^{-1})\right),\nonumber\\
&<&O(\log((1+h(b_n))e^{-1})).\label{G_A22}
\end{eqnarray}
From (\ref{G_A21}) and (\ref{G_A21}) we have
\begin{eqnarray}
\sup_{t_n<x}A_2(n,x)&<&2O(\log(1+h(b_n)))-2,\nonumber\\
&<&O(\log(1+h(b_n))).\label{G_A2}
\end{eqnarray}

Next, using the boundedness of $\lambda(x)$ and we have
\begin{eqnarray}
\sup_{x>t_n}A_3(n,x)&<&C\abs{n\bar{F}(b_n)-1},\nonumber\\
&=&C\abs{1+n\log F(b_n}.\nonumber
\end{eqnarray}
 From Theorem \ref{rate_G2}, $n\log F(b_n)<\log(e^{-1}(1+h(b_n))$ therefore 
\begin{eqnarray}
\sup_{x>t_n}A_3(n,x)&<&C\left(1+n\log F(b_n)\right),\nonumber\\
&<&O(\log(1+h(b_n))).\label{G_A3}
\end{eqnarray}
Finally, we have now proved uniform convergence over $x<t_n,$ and it suffices to show that
\[\sup_{x<t_n}\abs{g_n(x)-\lambda(x)}<\sup_{x<t_n}g_n(x)\vee \sup_{x<t_n}\lambda(x).\]
Since $F^{n-1}(\xi_n)=\exp\{(n-1)\log F(\xi_n)\}=\exp\{-n^{1/2}+n^{-1/2}\}.$ Therefore, if $\sup_{y\in\Real}f(y)<\infty$ then
\begin{eqnarray}
\sup_{x<t_n}g_n(x)&<&\sup_{y<\xi_n}nf(y)F^{n-1}(\xi_n)<\frac{Cne^{n^{-1/2}}}{e^{n^{1/2}}},\nonumber\\
&=&O\left(ne^{-n^{1/2}}\right).\nonumber
\end{eqnarray}
where, $y=a_n x+b_n.$ As well known that,
\begin{eqnarray}
\sup_{x<t_n}(\lambda(x))\leq \epsilon.\nonumber
\end{eqnarray}
where, $\epsilon>0.$ Therefore,
\begin{eqnarray}
\sup_{x<t_n}\abs{g_n(x)-\lambda(x)}&<&O\left(ne^{-n^{1/2}}\right).\label{G_A4}
\end{eqnarray}

From (\ref{G_A1}), (\ref{G_A2}), (\ref{G_A3}) and (\ref{G_A4}) we have
\begin{eqnarray}
\sup_{x\in\Real} \abs{g_n(x)-\lambda(x)}&<& O(h(b_n))+ O(\log(1+h(b_n)))+O\left(ne^{-n^{1/2}}\right).\nonumber
\end{eqnarray}
\end{proof}

\subsection{Fr\'{e}chet domain} Suppose $F \in \mathcal{D}(\Phi_\alpha),$ then $\lim_{n \rightarrow \infty} F^n(a_n x) = \Phi_{\alpha}(x),$ for $x>0,$ with $a_n = F^{\leftarrow}(1 - \frac{1}{n}) = \inf\{x: F(x) > 1 - \frac{1}{n}\},$ $n \geq 1\;$ and $b_n=0,$
  and  $1 - F$ is regularly varying so that
\[
\lim_{n \rightarrow \infty} \frac{\overline{F}(a_nx)}{\overline{F}(a_n)} = x^{- \alpha},\]
which is the same thing as saying that the function $L,$ defined by
\[L(x)=-x^{\alpha}\log F(x),\;\;x>0,\]
is slowly varying at $\infty.$ Now $L$ is said to be slowly varying with remainder function $h$ for large $n$ if 
\begin{eqnarray}\label{svg}
\frac{L(a_nx)}{L(a_n)}-1=O(h(a_n)), \;\;x>0,
\end{eqnarray}
where, $h(a_n)\to 0$ as $n\to\infty.$ Smith (1982), with simple argument show that
\begin{eqnarray}
-n\log F(a_nx)=x^{-\alpha}(1+O(h(a_n))).\label{F_svg}
\end{eqnarray}

From Theorem \ref{thm_von}, $F$ satisfies the von Mises condition.
A Von Mises condition guaranteeing $F\in\D(\Phi_\alpha)$ analogous to (\ref{von_F}) is
\begin{eqnarray}
h(t)=\frac{tf(t)}{F(t)\bar{F}(t)}-\alpha\to 0,\label{F_von2}
\end{eqnarray}
as $t\to\infty.$ We suppose there exists a nonincreasing continuous function $h$ and $h(x)=\sup_{t\geq x}\abs{h(t)}.$ We select $\delta$ so large that $h(\delta)<\alpha.$ Then for all $n$ such that $a_n^{-1}\delta<1,$  we assume 
\begin{eqnarray}\label{con_F}
\frac{h(a_nx)}{h(a_n)}\leq x^{-\rho},
\end{eqnarray}
where, $a_n^{-1}\delta\leq x\leq 1,$ and $\rho>0.$

\begin{thm}\label{F_den}
Suppose that (\ref{con_F}) holds. If $\sup_{x\in\Real}f(x)<\infty$ then for large $n$ 
\begin{eqnarray}
\sup_{x\in\Real} \abs{g_n(x)-\phi_\alpha(x)}&<&O(h(a_n))+ O\left(\frac{n}{e^{n^{1/2}}}\right).\nonumber
\end{eqnarray}
\end{thm}
\begin{proof}We have
\begin{eqnarray}
\sup_{x\in\Real}\abs{g_n(x)-\phi_\alpha(x)}<\sup_{x\in\Real}
\left(na_n f(a_nx)\abs{F^{n-1}(a_nx)-\Phi_\alpha(x)}
+\abs{na_n f(a_nx)\Phi_\alpha(x)-\phi_\alpha(x)}\right).\nonumber
\end{eqnarray}
Now we get,
\begin{eqnarray}
\sup_{x\in\Real}\abs{g_n(x)-\phi_\alpha(x)}&<&\sup_{x\in\Real}
\left(\frac{na_n f(a_nx)}{F(a_nx)}\abs{F^n(a_nx)-\Phi_\alpha(x)}+(\phi_\alpha(x)n a_n^{-\alpha}\abs{u(a_nx)-u(a_n)}\right)\nonumber\\
&&+\sup_{x\in\Real}(\phi_\alpha(x)\abs{na_n^{-\alpha}u(a_n)-1}),\nonumber\\
&=&\sup_{x\in\Real}[A_1(n,x)+A_2(n,x)+A_3(n,x)].\nonumber
\end{eqnarray}
where, $\abs{F^n(a_nx)-F(a_nx)\Phi_\alpha}$ is equivalent $\abs{F^n(a_nx)-\Phi_\alpha(x)}$ and $u(t)=\frac{t^{\alpha+1}}{\alpha}f(t),$ note that, from definition \ref{def1}, $f(t)$ is second order regular variation and then $u(t)$ is slowly varying.\\
Set, $\sup_{x\geq 1}A_1(n,x)=\sup_{x\geq 1}
\left(\frac{na_n f(a_nx)}{F(a_nx)}\abs{F^n(a_nx)-\Phi_\alpha(x)}\right).$ Then,
\begin{eqnarray}
\sup_{x\geq 1}A_1(n,x)&=&\sup_{x\geq 1}\left(\frac{a_n x f(a_nx)}{F(a_nx)\bar{F}(a_nx)}\frac{n\bar{F}(a_nx)}{x}\abs{F^n(a_nx)-\Phi_\alpha(x)}\right).\nonumber
\end{eqnarray}
From  (\ref{F_svg}) we have $\sup_{x\geq 1}(n\bar{F}(a_nx))<(1+O(h(a_n)))$ where, $\sup_{x\geq 1}x^{-\alpha}<1,$ and using Theorem \ref{rate_F},  $\sup_{x\geq 1}\abs{F^n(a_nx)-\Phi_\alpha(x)}<O(h(a_n)),$ and (\ref{F_von2}) gives, $\sup_{x\geq 1}\left(\frac{a_nxf(a_nx)}{F(a_nx)\bar{F}(xa_n)}\right)$ $<\alpha+h(a_n)$ then
\begin{eqnarray}
\sup_{x\geq 1}\abs{g_n(x)-\phi_\alpha(x)}&<&(\alpha+h(a_n))
(1+O(h(a_n)))O(h(a_n)),\nonumber\\
&<&O(h(a_n)).\label{F_A1}
\end{eqnarray}

Now, we choose $\xi_n$ by $-\log F(\xi_n)\sim n^{-1/2}$ so that $\xi_n\to\infty$ and $t_n=\frac{\xi_n}{a_n}.$ If $\frac{\xi_n}{a_n}\to c>0$ then $n^{1/2}\sim -n\log F(t_n a_n)\to c^{-\alpha}$ and this is contradict the fact that $n^{1/2}\to\infty$ therefore $\frac{\xi_n}{a_n}\to 0$ for large $n.$\\

Next, 
\begin{eqnarray}
\sup_{t_n\leq x\leq 1}A_1(n,x)&=&\sup_{t_n\leq x\leq 1}\left(\frac{a_nx f(a_nx)}{F(a_nx)\bar{F}(a_nx)x}n\bar{F}(a_nx)\abs{F^n(a_nx)-\Phi_\alpha(x)}\right).\nonumber
\end{eqnarray}
From Lemma \ref{rate_F} and Proposition 2.15, Page 112, Resnick (1987), for $\delta>0$ is chosen so large that $h(\delta)<\alpha$ if $h(a_nx)/h(a_n)\leq x^{-\beta}$ for $a_n^{-1}\delta\leq x\leq 1$ and $\beta>0$ then 
$$\sup_{a_n^{-1}\delta\leq x\leq 1}\abs{F^n(a_nx)-\Phi_\alpha(x)}<\beta^{-1}\theta\sup_{0<x< 1}\{x^{-1-\theta}(\log x)e^{-x^{-1}}\}h(a_n),$$ where, $\theta=\beta/(\alpha-h(\xi_n)),$
and 
$$\sup_{x\leq a_n^{-1}\delta}\abs{F^n(a_nx)-\Phi_\alpha(x)}<F^n(\delta)\vee \Phi_\alpha(a_n^{-1}\delta).$$

If $\delta a_n^{-1}< t_n\leq x\leq 1$ then from (\ref{F_svg}) $n\bar{F}(a_nx)=x^{-\alpha}(1+O(h(a_n))$ we have
\begin{eqnarray}
\sup_{t_n\leq x\leq 1}A_1(n,x)&<&(\alpha+h(a_n))h(a_n)
(1+O(h(a_n)))c(\alpha,\beta,\xi_n),\nonumber
\end{eqnarray}
where, $c(\alpha,\beta,\xi_n)=\beta^{-1}\theta\sup_{0<x<1}
\{x^{-2-\theta-\alpha}(\log x)e^{-x^{-1}}\}.$
Therefore,
\begin{eqnarray}
\sup_{t_n\leq x\leq 1}A_1(n,x)&<&O\left(h(a_n)\right).\label{F_A11}
\end{eqnarray}

If $t_n<x\leq\delta a_n^{-1}$ then 
\begin{eqnarray}
\sup_{t_n\leq x\leq \delta a_n^{-1}}A_1(n,x)&<&(\alpha+h(a_n))h(a_n)
(1+O(h(a_n)))(F^n(\delta)\vee\Phi_\alpha(a_n^{-1}\delta)),\nonumber\\
&=&O(h(a_n)).\label{F_A12}
\end{eqnarray}
From (\ref{F_A11}) and (\ref{F_A12}) we have
\begin{eqnarray}
\sup_{t_n\leq x\leq 1}A_1(n,x)&<&O\left(h(a_n)\right).\label{F_A2}
\end{eqnarray}
Hence, from (\ref{F_A1}) and (\ref{F_A2}) we have
\begin{eqnarray}
\sup_{t_n<x}A_1(n,x)<O(h(a_n)).\label{F_A5}
\end{eqnarray}

Now we set, $A_2(n,v)=\phi_\alpha(x)n a_n^{-\alpha}u(a_n)\abs{\frac{u(a_nx)-u(a_n)}{u(a_n)}}.$ Since $u(t)$ is slowly varying function and from (\ref{svg}) we have
\begin{eqnarray}
\sup_{x>t_n}A_2(n,x)&<&C_1\left(n a_nf(a_n)O(h(a_n))\right),\nonumber
\end{eqnarray}
where, $u(a_n)=\frac{1}{\alpha}a_n^{\alpha+1}f(a_n)$ and $C_1=\sup_{0<x}\phi_\alpha(x)\alpha^{-1}.$ From (\ref{F_svg}) and (\ref{F_von2}), for $x=1$ we have
\begin{eqnarray}
\sup_{x>t_n}A_2(n,x)&<&C_1\left(n\bar{F}(a_n)\frac{ a_nf(a_n)}{ F(a_n)\bar{F}(a_n)}O(h(a_n))\right),\nonumber\\
&=&C_1(1+O(h(a_n)))(\alpha +h(a_n))O(h(a_n)),\nonumber\\
&<&O(h(a_n)).\label{F_B}
\end{eqnarray}

Next, using the boundedness of $\phi_\alpha(x)$ and we have
\begin{eqnarray}
\sup_{x>t_n}A_3(n,x)=\sup_{x>t_n}\phi_\alpha(x)
\abs{na_n^{-\alpha}u(a_n)-1}.
\end{eqnarray}
From (\ref{F_svg}) and (\ref{F_von2}), for $x=1$ and  $u(a_n)=\frac{1}{\alpha}a_n^{\alpha+1}f(a_n)$ we have
\begin{eqnarray}
\sup_{x>t_n}A_3(n,x)&<&C_2\left(n\bar{F}(a_n)\frac{ a_nf(a_n)}{\alpha F(a_n)\bar{F}(a_n)}-1\right),\nonumber\\
&=&C_2\left((1+O(h(a_n)))\left(1 +\frac{h(a_n)}{\alpha}\right)-1\right),\nonumber\\
&<&O(h(a_n)).\label{F_C}
\end{eqnarray}
where, $C_2=\sup_{x>t_n}\phi_\alpha(x)\alpha^{-1}.$

Finally, We have now proved uniform convergence over $x<t_n,$ and it suffices to show that
\[\sup_{x\leq t_n}\abs{g_n(x)-\phi_\alpha(x)}<\sup_{x\leq t_n}g_n(x)\vee \sup_{x\leq t_n}\phi_\alpha(x).\]
Since $F^{n-1}(\xi_n)=\exp\{(n-1)\log F(\xi_n)\}=\exp\{-n^{1/2}+n^{-1/2}\}.$ Therefore, If $\sup_{y\in\Real}f(y)<\infty$ then
\begin{eqnarray}
\sup_{x\leq t_n}g_n(x)&<&\sup_{y\leq \xi_n}nf(y)F^{n-1}(\xi_n)<\frac{Cne^{n^{-1/2}}}{e^{n^{1/2}}},\nonumber\\
&=&O\left(\frac{n}{e^{n^{1/2}}}\right).\nonumber
\end{eqnarray}
where, $y=a_n x.$ As well known that,
\begin{eqnarray}
\sup_{x\leq t_n}(\phi_\alpha(x))\leq  \sup_{0<x\leq t_n}\phi_\alpha(x)=\epsilon.\nonumber
\end{eqnarray}
Therefore,
\begin{eqnarray}
\sup_{x\leq t_n}\abs{g_n(x)-\phi_\alpha(x)}&<&O\left(\frac{n}{e^{n^{1/2}}}\right).\label{F_A3}
\end{eqnarray}

From (\ref{F_A5}), (\ref{F_B}), (\ref{F_C}) and (\ref{F_A3}) we have
\begin{eqnarray}
\sup_{x\in\Real} \abs{g_n(x)-\phi_\alpha(x)}&<&O(h(a_n))+ O\left(\frac{n}{e^{n^{1/2}}}\right).\nonumber
\end{eqnarray}
\end{proof}

\section{Rates of convergence for R\'{e}nyi entropy of max domain of attraction}
Let $M_n$ is a sequence of random variables with density function $g_n$ and $Y$ is a random variable with density function $g.$ There exists a large $N$ and $\epsilon>0$ such that $\sup_{x\in\Real}\abs{g_n(x)-g(x)}<\epsilon.$ Now, we prove an inequality which is use consequently in our main results.

\begin{thm}\label{rate_lem1} If $\int_{\Real}(g(x))^\beta dx<\infty$ and $\int_{\Real}(f(x))^{\beta-1} dx<\infty$ for $\beta>1$ and $M>0$ then
\begin{eqnarray}
\abs{H_\beta(g_n)-H_\beta(g)}&<&M\sup_{x\in\Real}\abs{g_n(x)-g(x)}.\nonumber
\end{eqnarray}
\end{thm}
\begin{proof} From (\ref{Renyi_entropy}) we have
\begin{eqnarray}
\abs{H_\beta(g_n)-H_\beta(g)}&=&\frac{1}{\abs{1-\beta}}\abs{\log\int_{\Real}g_n(x)^{\beta}dx
-\log\int_{\Real}g(x)^\beta dx},\nonumber\\
&<&\frac{1}{\abs{1-\beta}}\log\left(\frac{\int_{\Real}\abs{g_n(x)^\beta-g(x)^\beta}dx}
{\int_{\Real}g(x)^\beta\,dx}+1\right),\nonumber
\end{eqnarray}
Using Taylor expansion of logarithm, $\frac{\int_{\Real}\abs{g_n(x)^\beta-g(x)^\beta} dx}{\int_{\Real}g(x)^\beta dx}<1$ and $\abs{1-\beta}\int_{\Real}g(x)^\beta dx<M_1$ then 
\begin{eqnarray}
\abs{H_\beta(g_n)-H_\beta(g)}&<&M_1\int_{\Real}\abs{g_n(x)^{\beta}-g(x)^{\beta}}dx,\nonumber\\
&<&M_1\int_{\Real}\abs{g_n(x)-g(x)}\abs{g_n(x)-g(x)}^{\beta-1}dx,\nonumber\\
&<&M_1\sup_{x\in\Real}\abs{g_n(x)-g(x)}\int_{\Real}\abs{g_n(x)-g(x)}^{\beta-1}dx.\label{gen_thm}
\end{eqnarray}
From Theorem (\ref{resnick_1}) if $\int_{\Real}(f(x))^{\beta-1}<\infty$ then for large $n$ and $M_2>0$ such that $\int_{\Real}\abs{g_n(x)-g(x)}^{\beta-1}$ $<M_2$ and using (\ref{gen_thm}) we have,
\begin{eqnarray}
\abs{H_\beta(g_n)-H_\beta(g)}&<&M\sup_{x\in\Real}\abs{g_n(x)-g(x)}.\nonumber
\end{eqnarray}
where, $M=M_1M_2.$
\end{proof}

\begin{thm}
Suppose that (\ref{con_F}) hold. If $\sup_{x\in\Real}f(x)<\infty$ and $f$ is decreasing function near $\infty$ and $\int_{\Real}(f(x))^\beta dx<\infty$ then for large $n$ and $\beta> 1$
\begin{eqnarray}
\abs{H_\beta(g_n)-H_\beta(\phi_\alpha)}\leq  O(h(a_n))+ O\left(\frac{n}{e^{n^{1/2}}}\right).\nonumber
\end{eqnarray}
\end{thm}
\begin{proof}
From Theorem \ref{rate_lem1} we have
\begin{eqnarray}
\abs{H_\beta(g_n)-H_\beta(g)}&<&M\sup_{x\in\Real}\abs{g_n(x)-g(x)}.\nonumber
\end{eqnarray}
If Theorem \ref{ren_con1} and Theorem \ref{F_den} hold then for $1\leq \beta$ we have
\begin{eqnarray}
\abs{H_\beta(g_n)-H_\beta(\phi_\alpha)}&<&M\left(O(h(a_n))+ O\left(\frac{n}{e^{n^{1/2}}}\right)\right),\nonumber\\
&<&\left(O(h(a_n))+ O\left(\frac{n}{e^{n^{1/2}}}\right)\right).\nonumber
\end{eqnarray}
\end{proof}

\begin{thm}
Suppose $F\in\D(\Lambda)$ and there exist $k>0$ such that for $n\geq n_0$
\begin{eqnarray}
\frac{u(b_n)}{b_nh(b_n)}\leq k.\nonumber
\end{eqnarray}
and for $c<k^{-1}$ and $b>0,$ $d>0,$ 
\begin{eqnarray}
\frac{h(b_n(1-ck))}{h(b_n)}\leq d(1-ck)^{-b}.\nonumber
\end{eqnarray}
and (\ref{G_h}) holds. If $\sup_{x\in\Real}f(x)<\infty,$ $f$ is decreasing near $r(F)$ and $\int_{\Real}(f(x))^\beta dx<\infty$ then for large $n$ and $\beta>1$ 
\begin{eqnarray}
\abs{H_\beta(g_n)-H_\beta(\lambda)}<O(h(b_n))+ O(\log(1+h(b_n)))+O\left(ne^{-n^{1/2}}\right).
\nonumber
\end{eqnarray}
\end{thm}

\begin{proof}
From Theorem \ref{rate_lem1} 
\[\abs{H(g_n)-H(\phi_\alpha)}<M\sup_{x\in\Real}\abs{g_n(x)-\lambda(x)}.\]
If Theorem \ref{ren_con2} and Theorem \ref{con_G} hold then
\begin{eqnarray}
\abs{H_\beta(g_n)-H_\beta(\lambda)}<M\left(O(h(b_n))+ O(\log(1+h(b_n)))+O\left(ne^{-n^{1/2}}\right)\right),\nonumber\\
<O(h(b_n))+ O(\log(1+h(b_n)))+O\left(ne^{-n^{1/2}}\right).\nonumber
\end{eqnarray}
\end{proof}

\appendix
\section{}\label{more results}
\begin{defn}(De Haan and Resnick (1996))\label{def1}
We suppose $F$ is twice differentiable. Recall that $f$ is second regularly varying with index $\alpha$ if for all $x>0,$
\[\lim_{t\to\infty}\frac{f(tx)}{f(t)}\;=\;x^{-\alpha-1}.\]
\end{defn}

\begin{thm}(Proposition 1.15 and 1.16, Resnick (1987))\label{thm_von}
Suppose that distribution function $F$ is absolutely continuous with density $f$ which is eventually positive. \\
\item[(a)] If for some $\alpha>0$
\begin{eqnarray}\label{von_F}
\lim_{x\to\infty}\dfrac{xf(x)}{\overline{F}(x)}=\alpha,
\end{eqnarray}
then $F\in\D(\Phi_\alpha).$\\
\item[(b)] If $f$ is nonincreasing and $F\in\D(\Phi_\alpha)$ then (\ref{von_F}) holds.
\end{thm}

\begin{thm}\label{thm_g_von}(Proposition 1.17, Resnick (1987)) Let  $F$ be absolutely continuous in a left neighborhood of $r(F)$ with density $f.$ If 
\begin{eqnarray} \label{lem3_appendix} \lim_{x\uparrow r(F)}f(x)\int_{x}^{r(F)}\overline{F}(t)dt/\overline{F}(x)^2=1,\end{eqnarray} 
 then $F\in\D(\Lambda).$ In this case we may take,
\[u(x)=\int_{x}^{r(F)}\overline{F}(t)dt/\overline{F}(x), \;b_n=F^{\leftarrow}(1-1/n),\;\;a_n=u(b_n).\]
\end{thm}

\begin{lem}\label{rate_F}(Lemma 2.14, Resnick (1987)) Suppose that $\frac{h(a_nx)}{h(a_n)}\leq x^{-\rho}$ for $a_n^{-1}\delta\leq x\leq 1$ and $\rho>0.$ Let $\delta>0$ is chosen so large that $h(\delta)<\alpha.$ Then for $n$ such that $a_n^{-1}\delta<1$
\[\sup_{\delta a_n^{-1}\leq x\leq 1}\left(\abs{\Phi_{\alpha-h(a_nx)}(x)-\Phi_\alpha(x)}\vee\abs{\Phi_{\alpha+h(a_nx)}(x)-\Phi_\alpha(x)}\right)=c(\alpha,\beta,\delta)O(h(a_n)),\]
where,
$c(\alpha,\beta,\delta)=\beta^{-1}\theta\sup_{0<y<1}\{y^{-1-\theta}(\log y)e^{-y^{-1}}\},$ and $\theta=\beta/(\alpha-h(\delta)).$
\end{lem}
\begin{thm}\label{rate_G1}(Proposition 2.18, Resnick (1987))
If (\ref{G_h}) hold and $h(x)=\sup_{x<y}h(y)$ then
\[\sup_{x\geq 0}\abs{F^n(a_nx+b_n)-\Lambda(x)}\leq e^{-1}h(b_n).\]
\end{thm}

\begin{thm} \label{rate_G3}(Proposition 2.20, Resnick (1987))
Suppose $F\in\D(\Lambda)$ and there exist $k>0$ such that for $n\geq n_0$
\begin{eqnarray}
\frac{u(b_n)}{b_nh(b_n)}\leq k.\label{G_e2}
\end{eqnarray}
and for $c<k^{-1}$ and $b>0,$ $d>0,$ 
\begin{eqnarray}\label{G_e3}
\frac{h(b_n(1-ck))}{h(b_n)}\leq d (1-ck)^{-b}.
\end{eqnarray}
then
\[\sup_{x\in\Real}\abs{F^n(a_nx+b_n)-\Lambda(x)}<O(h(b_n)).\]
\end{thm}

\begin{thm}\label{rate_G2}(Proposition 2.17, Resnick (1987)) For a positive real number $z$ define the distribution functions
\begin{eqnarray*}
&&F(z,x)  = \left\lbrace	
			\begin{array}{l l}
		     0, &\;\;\; x< -z^{-1}, \\
			 \exp(-(1+zx)^{-z^{-1}}), &\;\;\; x>-z^{-1};\\
			 \end{array}
		 \right. \\					
&&F(-z,x) =  \left\lbrace
\begin{array}{l l} \exp(- (1-zx)^{z^{-1}}), & x<z^{-1}, \\
					1, & x>z^{-1};
		\end{array}\right. \\
	\end{eqnarray*}
Then for $0<z<1$
\[\sup_{x\in\Real}\abs{F(\pm z,x)-\Lambda(x)}\leq e^{-1}z.\]
\end{thm}

\begin{thm}\label{resnick_1}(Theorem 3, de Haan and Resnick (1982)) Let $\{M_n, n\geq 1\}$ be iid with df $F$ wich is absolutely continuous with density $f.$ Set $M_n=\vee_{i=1}^{n}X_i$ and let $g_n$ be the density of the normalized $M_n.$
\item[(a).] Suppose (\ref{von_F}) holds. If $\int_{\Real}(f(x))^\beta dx<\infty$ and $\beta>(1+\alpha)^{-1}$ then
\[\int_{\Real}\abs{g_n(x)-\phi_\alpha(x)}^\beta dx\to 0,\]
as $n\to\infty.$
\item[(b).] Suppose (\ref{lem3_appendix}) holds. If $\int_{\Real}(f(x))^\beta<\infty$ then
\[\int_{\Real}\abs{g_n(x)-\lambda(x)}^\beta dx\to 0,\]
as $n\to\infty.$
\end{thm}

\begin{thm}\label{ren_con1}(Theorem 3.1, Saeb (2014)) Suppose $F\in D(\Phi_\alpha)$ is absolutely continuous with pdf $f$ which is eventually positive and decreasing in left neighborhood of $\infty.$ If $\int_{\Real}(f(x))^\beta dx<\infty$ then for $\beta>1$
\[\lim_{n\to\infty}H_\beta(g_n)=H_\beta(\phi_\alpha).\]
\end{thm}

\begin{thm}\label{ren_con2}(Theorem 3.3, Saeb (2014)) Suppose $F\in D(\Lambda)$ is absolutely continuous with pdf $f$ which is eventually positive and decreasing in left neighborhood of $r(F).$ If $\int_{\Real}(f(x))^\beta dx<\infty$ then for $\beta>1$
\[\lim_{n\to\infty}H_\beta(g_n)=H_\beta(\lambda).\]
\end{thm}

\end{document}